\documentclass[11pt]{article}
\usepackage{amsmath,amsthm,amscd,amsfonts,amssymb,mathrsfs}
\usepackage{tikz}
\usepackage{graphicx}

\allowdisplaybreaks
\newtheorem{theorem}{Theorem}[section]
\newtheorem{lemma}[theorem]{Lemma}
\newtheorem{proposition}[theorem]{Proposition}
\newtheorem{remark}[theorem]{Remark}
\numberwithin{equation}{section}

\newcommand{\ZZ}{\mathbb{Z}}

\input xy
\xyoption{all}

\title{Spectral statistics of  random Schr\"{o}dinger operator with growing potential}
\author{Anish Mallick\footnote{E-mail:\texttt{anish.mallick@icts.res.in}, Institute: ICTS-TIFR Bangaluru, India.} \& Dhriti Ranjan Dolai\footnote{E-mail:\texttt{dhriti\_vs@isibang.ac.in}, Institute: ISI Bangalore, India.}}
\date{\today}
\begin{document}
\setlength{\belowdisplayskip}{2pt} \setlength{\belowdisplayshortskip}{2pt}
\setlength{\abovedisplayskip}{2pt} \setlength{\abovedisplayshortskip}{2pt}
\maketitle
\noindent {\bf Abstract.} In this work we investigate the spectral statistics of random Schr\"{o}dinger operators 
$H^\omega=-\Delta+\sum_{n\in\mathbb{Z}^d}(1+|n|^\alpha)q_n(\omega)|\delta_n\rangle\langle\delta_n|$, $\alpha>0$
acting on $\ell^2(\mathbb{Z}^d)$ where  $\{q_n\}_{n\in\mathbb{Z}^d}$ are i.i.d random variables distributed uniformly on $[0,1]$.\\\\
{\bf AMS 2000 MSC}: 35J10, 81Q10, 35P20.

\section{Introduction}
\noindent We consider the random Scr\"{o}dinger operators with unbounded potentials given by
\begin{equation}
 \label{Model}
(H^\omega u)(n)= \sum_{\| n-m \|_1=1}(u(n)-u(m))+(1+\lVert n\rVert^\alpha_2)q_n(\omega)u(n)\qquad\forall n\in\ZZ^d,
\end{equation}
for $u\in\ell^2(\mathbb{Z}^d)$ with finite support. 
We will use the notation $\lVert\cdot\rVert_p$ to denote the standard $p$-norm and $\{q_n\}_n$ are independent identically distributed random variables distributed uniformly on $[0,1]$. 
We will view $q_n(\cdot)$ as a random field on the probability space $(\Omega,\mathbb{P},\mathcal{B}(\Omega))$.
Denote
\begin{align}
 \label{lap}
(H_0 u)(n):=((2d-\Delta)u)(n)=\sum_{\lVert n-m \rVert_1=1} (u(n)-u(m))\qquad\forall u\in\ell^2(\ZZ^d),n\in\ZZ^d,
\end{align}
and observe that  $0\leq H_0\leq 4d$.
Henceforth we will use $V_n(\omega)=b_nq_n(\omega)$ where $b_n=(1+\lVert n\rVert_2^\alpha)$ and following Dirac notation we have
\begin{equation*}
 V(\omega)=\sum_{n\in\mathbb{Z}^d}V_n(\omega)|\delta_n\rangle\langle\delta_n|,
\end{equation*}
so one can express $H^\omega$ as
\begin{equation}
\label{model}
 H^\omega=2d-\Delta+V(\omega).
\end{equation}
Here we will concentrate on the case $0<\alpha<d$. This is because the spectrum of $H^\omega$ is almost surely discrete for $\alpha>d$.

\noindent The spectral theory for the operator \eqref{Model} were studied by Gordon-Molchanov-Tsagani \cite{GMT} for one-dimension and Gordon-Jak\v{s}i\'{c}-Mochanov-Simon \cite{GJMS} for higher dimension.
We require some results from \cite{GJMS} to formulate our problem. The form $\mathscr{D}$ associated to $H_0$ is defined by 
\begin{equation}
 \label{form}
\mathscr{D}(\phi)=\sum_{ \|n-m\|_1=1}|\phi(n)-\phi(m)|^2,~~\phi\in\ell^2(\mathbb{Z}^d).
\end{equation}
Denote
\begin{equation}
\label{def2}
 a_k=\inf_{A_k}\inf_{\substack{\|\phi\|=1\\ supp(\phi)\subseteq A_k}} \mathscr{D}(\phi),
\end{equation}
where $A_k$ are connected (there is a path between every two points) subsets of $\mathbb{Z}^d$ with $\#A_k=k$.
Then $\{a_k\}_{k\ge 0}$ (set $a_0=\infty$) is a strictly decreasing sequence of positive numbers.
We denote $H^\omega_{\Lambda_L}$ to be the restriction of $H^\omega$ on the subspace $\ell^2(\Lambda_L)$ where $\Lambda_L$ is the cube $\{(n_1,n_2,\cdots,n_d)\in \mathbb{Z}^d:|n_i|\leq L\}$, and set
\begin{equation*}
 N^\omega_L(E):=\#\big\{j: E_j\leq E,~E_j\in\sigma(H^\omega_{\Lambda_L})\big\},\qquad E\in\mathbb{R}.
\end{equation*}
Following results are recalled from Gordon-Jak\v{s}i\'{c}-Molchanov-Simon \cite{GJMS}.
\begin{theorem}[{\cite[Theorem 1.2]{GJMS}}]
\label{refth1}
 If $\frac{d}{k}\ge \alpha>\frac{d}{k+1}$ for positive integer $k$, for a.e $\omega$ 
\begin{enumerate}
\item[(i)] $\sigma(H^\omega)=\sigma_{pp}(H^\omega)$ and eigenfunctions of $H^\omega$ decay at least exponentially,
\item[(ii)] $\sigma_{ess}(H^\omega)=[a_k,\infty),$
\item[(iii)] $\#\sigma_{disc}(H^\omega)<\infty$. 
\end{enumerate}
\end{theorem}
\noindent and
\begin{theorem}[{\cite[Theorem 1.4]{GJMS}}]
\label{refth2}~
\begin{enumerate}
 \item[(i)] If $\frac{d}{k}> \alpha>\frac{d}{k+1}$ and $E\in(a_j,a_{j-1})$ for $1\leq j\leq k$, then
$$\lim_{L\to\infty} \frac{N_L^\omega(E)}{L^{d-j\alpha}}=N_j(E)$$
exist for a.e $\omega$ and is a non-random function.
\item[(ii)] If $\alpha=\frac{d}{k}$ and $E\in(a_j,a_{j-1})$,~$1\leq j<k$, the above is valid. If $E\in (a_k, a_{k-1})$ then
$$\lim_{L\to\infty} \frac{N_L^\omega(E)}{ln L}=N_k(E)$$
exists for a.e $\omega$ and is a non-random function.
\end{enumerate}
In both cases, $N_j$ is a continuous function on $(a_j, a_{j-1})$, and
$$N_j(E)\sim D_j (E-a_j)^j~~as~~E\to a_j.$$
The constants $D_j$ are of combinatorial nature.
\end{theorem}

\noindent Here we will study the eigenvalue statistics for the operator $H^\omega$ in the essential spectrum.
By Theorem \ref{refth1} (i), the essential spectrum of $H^\omega$ is pure point, so one can expect that the statistics is Poisson point process, which we will show here.
Optimal Wegner estimate is absent for general intervals except for any interval contained in $(a_1,\infty)$. 
So in the interval $(a_1,\infty)$ we will show that the local eigenvalue statistics as developed by Minami \cite{NM} is simple Poisson point process.
On the other hand for a point process defined on the intervals $\{(a_j,a_{j-1})\}_{j=2}^k$, we will focus on the unfolded eigenvalue statistics as developed by Klopp\cite{Klopp}.
Based on \cite[Lemma 5.2]{GJMS}, we are able to show the absolute continuity of $N_j$ in the region $(a_j, a_{j-1})$ for $2\leq j\leq k$.
Above theorem also gives the asymptotic expression of $N_j$ as $E\rightarrow a_j$, which implies $N_j(a_{j-1})-N_j(a_j)>0$.
Hence the distribution defined by $x\mapsto \frac{N_j(x)-N_j(a_{j})}{N_j(a_{j-1})-N_j(a_j)}$ is well-defined and non-trivial for $x\in(a_j,a_{j-1})$.
\\\\
But on the interval $(a_1,\infty)$, we have usual Wegner estimate which provides the absolute continuity of $N_1$. 
We also have $N_1^\prime(x)>0,~\forall~x>a_1$, hence the local eigenvalue statistics as developed by Minami \cite{NM} can also be computed.
\\\\
Before we state  our result for the unfolded eigenvalue statistics, we require few notations.
For $2\leq j\leq k$, set $N_j(a_j, a_{j-1}):= N_j({a_{j-1}})-N_j(a_j)$ and define $\tilde{N}_j(x)=\frac{N_j(x)-N_j(a_{j})}{N_j(a_j, a_{j-1})}$. Setting
$\beta_{j, L}=N_j(a_j, a_{j-1}) L^{d-j\alpha}$, 
define the point process $\{\xi^{\omega,t}_{L,j}(\cdot)\}_L$ by
\begin{equation}
\label{unfold-pt}
 \xi^{\omega,t}_{L,j}(\cdot)=\sum_{x\in\sigma\big(H^\omega_{\Lambda_L}\big)\cap(a_j, a_{j-1})}\delta_{\beta_{j, L}[\tilde{N}_j(x)-t]}(\cdot),~~
 2\leq j\leq k.
\end{equation} 
for $t\in(0,1)$. We will view $t$ as a random variable which is uniformly distributed on $(0,1)$, so $\xi_{L,j}^{\cdot,\cdot}$ is a 
random measure define on the probability space $(\Omega\times[0,1],\mathcal{B}(\Omega)\otimes\mathcal{B}_{[0,1]}, \mathbb{P}\times Leb)$, where $Leb$ is the Lebesgue measure on $[0,1]$.
With these definitions in place we state our main results.
\begin{theorem}
 \label{thm1}
For the operator $H^\omega$ defined as \eqref{Model}, the point process $\xi^{\omega,t}_{L,j}(\cdot)$, 
defined on the probability space $(\Omega\times[0,1],\mathcal{B}(\Omega)\otimes\mathcal{B}_{[0,1]},\mathbb{P}\times Leb)$, 
converges weakly to Poisson point process with intensity as the Lebesgue measure on $\mathbb{R}$, for $2\leq j\leq k$.
\end{theorem}

\noindent It should be noted that above theorem can be easily extended to $j=1$ case with slight modification. 
But we study the local eigenvalue statistics as developed by Minami \cite{NM}.
On the interval $(a_1, \infty)$ we can study the local eigenvalue statistics associated with $H^\omega_{\Lambda_L}$. 
Define the point process
\begin{equation}
 \label{defpoint}
\xi^\omega_{L,E}(\cdot)=\sum_{x\in\sigma(H^\omega_{\Lambda_L})}\delta_{L^{d-\alpha}(x-E)}(\cdot), 
\end{equation}
for $E\in (a_1,\infty)$.
From Theorem \ref{refth2} it is clear that the average spacing between eigenvalues of $H^\omega_{\Lambda_L}$
in the region $(a_1,\infty)$ is of order $\frac{1}{L^{d-\alpha}}$.
Our goal is to show that the weak limit of the sequence of point processes $\{\xi^\omega_{L,E}\}_L$
is Poisson point process with intensity measure $N^\prime_1(E)dx$. 
But we can provide an explicit expression of $N_1(x)$ as done in the following:
\begin{theorem}\label{thmReg1}
The function $N_1$ is given by
$$N_1(x)=C_{d,\alpha}(x-2d)\qquad x>2d, $$
where 
$$C_{d,\alpha}=\lim_{L\rightarrow\infty}\frac{1}{L^{d-\alpha}}\sum_{n\in\Lambda_L}\frac{1}{b_n}.$$
\end{theorem}

\noindent Note that the theorem below is slightly more stronger than Theorem \ref{thm1}. 
Here the point process converges for each $E$, but in the previous theorem we viewed $E$ as a random variable.

\noindent We can compute the limit of $\frac{1}{L^{d-\alpha}}N_L^\omega(E)$ by approximating it with a nice family of functions and show uniform convergence over compact sets, 
hence the local density of states converges compact uniformly.
So for the statistics of $\{\xi^\omega_{L,E}\}_L$, we have
\begin{theorem}
 \label{main}
For $0<\alpha< d$, let $H^\omega$ be defined by \eqref{Model}  and set  $\xi^\omega_{L,E}$ to be the point measure \eqref{defpoint}. 
Then for $E\in (8d^2, \infty)$, the sequence of point process $\{\xi^\omega_{L,E}\}_L$ converges weakly to the Poisson point process with intensity measure $C_{d,\alpha}dx$
(Lebesgue measure). So for any bounded Borel set $B\in\mathbb{B_{\mathbb{R}}}$
$$\lim_{L\to\infty}\mathbb{P}\big(\omega:\xi^\omega_{L,E}(B)=n \big)=e^{-C_{d,\alpha}|B|}\frac{\big(C_{d,\alpha}|B|\big)^n}{n!}~~for~~n\in\mathbb{N}\cup\{0\}.$$
\end{theorem}
\noindent The study of eigenvalue statistics were done by Molchanov \cite{Mol} in one-dimension and by Minami \cite{NM} for Anderson model at high disorder with stationary potentials.
 In \cite{MD} Dolai-Krishna obtained Poisson statistics for discrete Anderson model with uniformly H\"{o}lder continuous single site potential.
Eigenfunction statistics were studied by Nakano \cite{Na} for continuous model and Killip-Nakano \cite{KN} for lattice case with bounded density for single site distribution.
Dolai-Mallick \cite{AD} studied the eigenfunction statistics for H\"{o}lder continuous single site distribution on lattice.
\\\\
But the energy level statistics as developed by Molchanov and Minami usually require the presence of absolute continuity for density of states and Wegner estimate with correct scaling.
Minami \cite{NM1} (see also Berry-Tabor \cite{BT}) conjectured a different type of statistics using the unfolded eigenvalues for discrete Anderson model in the region of localization.
Klopp\cite{Klopp} proved the Poisson limit theorem for unfolded eigenvalues of random operators in localized regime. 
The proof does not needs strong regularity assumption on the Integrated density of states.
The work in \cite{Klopp} have some connection with  Germinet-Klopp \cite{GK1}, which studied various statistics related to the eigenvalues and eigenfunctions of random Hamiltonians $H^\omega$ in the localized regime.
One of the typical result in \cite{GK1} is convergence of empirical level spacing distribution, i.e
$$DSL(x,\omega,\Lambda) \xrightarrow[|\Lambda|\to\infty]{Uniformly} g(x),$$
where 
$$DSL(x,\omega,\Lambda)=\frac{\#\big\{j:E_{j+1}(\omega,\Lambda)-E_j(\omega,\Lambda) \ge x\big\}}{|\Lambda|}~~\&~~ g(x)=\int_{\Sigma}e^{-\nu(E)x}\nu(E)dE.$$ 
In the above, $\{E_j(\omega,\Lambda)\}$ are the eigenvalues of $H^\omega_{\Lambda}$, $\nu$ is the density of states and $\Sigma$ is the almost sure spectrum of $H^\omega$.
\\\\
For decaying potentials (i.e $\alpha<0$) the spectral statistics were studied on several occasions.
Killip-Stoiciu \cite{KS} studied the CMV matrices whose matrix elements decay like $n^{-\alpha}$. They showed that,
for (i) $\alpha>1/2$ statistics is the clock, (ii) $\alpha=1/2$ limiting process is
circular $\beta-$ensemble, (iii) $0<\alpha<1/2$ the statistics is Poisson.  Analogues of Killip-Stoiciu\cite{KS} was done by Kotani-Nakano\cite{KotNak} for
the one-dimensional Schr\"{o}dinger operator with decaying potentials in the continuum model and obtained the same statistics for $\alpha>1/2$ and $\alpha=1/2$.
Krichevski-Valk\'{o}-Vir\'{a}g \cite{VV} studied the one-dimensional discrete Schr\"{o}dinger operator with the random potential decaying like $n^{-1/2}$ and obtained the Sine$\beta$-process. 
In \cite{DM} Dolai-Krishna consider the Anderson Model with decaying Random Potentials and shown that statistics inside $[-2d, 2d]$ for dimension $d\ge 3$ is independent of the randomness and agrees with that of the free part $\Delta$.

\section{\bf Preliminaries}
For any $\Lambda \subset\mathbb{Z}^d$ we consider the canonical orthogonal projection $\chi_\Lambda$ onto $\ell^2(\Lambda)$ and define the matrices
\begin{equation}
\label{defgr}
 H^{\omega}_\Lambda=\big(\langle\delta_n, H^{\omega}\delta_m\rangle\big)_{n,m\in \Lambda},~
G^\Lambda(z;n,m)=\langle\delta_{n},(H_\Lambda^{\omega}-z)^{-1}\delta_{m}\rangle,~G^\Lambda(z)=(H_\Lambda^{\omega}-z)^{-1}.
\end{equation}
$$
G(z)=(H^{\omega}-z)^{-1},~~G(z;n,m)=\langle\delta_{n},(H^{\omega}-z)^{-1}\delta_{m}\rangle, ~~z\in\mathbb{C}^{+}.
$$
Note that $H^{\omega}_\Lambda$ is the matrix 
$$\chi_\Lambda H^{\omega}\chi_\Lambda~:~\ell^2(\Lambda)\longrightarrow\ell^2(\Lambda),~ a.e ~\omega.$$
\noindent Define
\begin{equation}
\label{finite}
 \nu_L(\cdot)=\frac{1}{L^{d-\alpha}}\displaystyle\sum_{n\in\Lambda_L}\mathbb{E}^\omega\big(\langle \delta_n, E_{H^\omega_{\Lambda_L}}(\cdot)\delta_n\rangle\big),
\end{equation}
 Observe 
\begin{equation}
\label{fact}
\displaystyle\sum_{n\in\Lambda_L}b_n^{-1} =\displaystyle\sum_{n\in\Lambda_L}(1+\lVert n\rVert_2^\alpha)^{-1}=\Theta(L^{d-\alpha}).
\end{equation}
For any bounded interval $I=[a,b]$ we have
$$\nu_L(I)=\frac{1}{L^{d-\alpha}}\mathbb{E}^\omega[N^\omega_L(b)-N^\omega_L(a)].$$

\noindent In this section we calculate the Minami and Wegner estimate and show the positivity of the measure $\nu$ given by
\begin{equation}
\label{defgama}
N_1(x)=\nu(a_1,x),~for~x>a_1,
\end{equation}
here $N_1$ is defined as in (i) of Theorem \ref{refth2}.\\
We start with spectral averaging estimate for the model (\ref{Model}).
\begin{proposition}
\label{pro1}
Let $I=[a,b]\subset\mathbb{R}$ be a bounded interval and $\Lambda\subseteq \mathbb{Z}^d$ 
\begin{center}
$\mathbb{E}^{\omega}(\langle\delta_n,E_{H^\omega_\Lambda}(I)\delta_n\rangle )\leq \pi(1+\lVert n\rVert_2^\alpha)^{-1}|I|,\qquad n\in \Lambda$.
\end{center}
\end{proposition}
\begin{proof}
 Write $H^{\omega}_\Lambda$ as 
\begin{align*}
 H^{\omega}_\Lambda &=\left(\chi_{_\Lambda}H_0\chi_{_\Lambda}+\sum_{n\neq k\in \Lambda}b_kq_k(\omega)|\delta_k\rangle\langle\delta_k|\right)+ b_nq_n(\omega)|\delta_n\rangle\langle\delta_n|\\
 &=H^{\omega/ n}_\Lambda+b_nq_n(\omega)|\delta_n\rangle\langle\delta_n|.
\end{align*}
Using resolvent equation
\begin{equation*}
\langle\delta_{n},(H^{\omega}_\Lambda-z)^{-1}\delta_{n}\rangle =\frac{1}{b_{n}q_{n}(\omega)+\langle\delta_{n},(H^{\omega/ n}_\Lambda-z)^{-1}\delta_{n}\rangle^{-1}}.
\end{equation*}
For $z\in\mathbb{C}^{+}$, one has $\langle\delta_{n},(H^{\omega/ n}_\Lambda-z)^{-1}\delta_{n}\rangle\in \mathbb{C}^{+}$ so calling $\langle\delta_{n},(H^{\omega/ n}_\Lambda-z)^{-1}\delta_{n}\rangle^{-1}=A+\iota B$, we get 
\begin{center}
$\int_{\mathbb{R}}Im \langle\delta_{n},(H^{\omega}_\Lambda-z)^{-1}\delta_{n}\rangle d\mu(q_n)
=\frac{1}{b_n}\int_{\mathbb{R}}\frac{B}{(x+A)^2+B^2}d\mu(x)
\leq \pi b_n^{-1}$.
\end{center}
The following is immediate from above,
\begin{equation}
 \label{pro1r}
\mathbb{E}^{\omega}\big(Im \langle\delta_{n},(H^{\omega}_\Lambda-\sigma-i\tau)^{-1}\delta_{n}\rangle\big)
\leq \pi b_n^{-1}.
\end{equation}
Using Fubini theorem it follows
\begin{center}
$\mathbb{E}^{\omega}\bigg(\int_a^b Im\langle\delta_{n},(H^{\omega}_\Lambda-\sigma-i\tau)^{-1}\delta_{n}\rangle d\sigma\bigg)
\leq \pi b_n^{-1}|I|$.
\end{center}
Using Stone's formula \cite[Theorem VII.13]{RS} 
\begin{center}
 $\frac{\pi}{2}\mathbb{E}^{\omega}[\langle\delta_{n},E_{H^{\omega}_\Lambda}[a,b]\delta_{n}\rangle+\langle\delta_{n},E_{H^{\omega}_\Lambda}(a,b)\delta_{n}\rangle] \leq\pi  b_n^{-1}|I|$.
\end{center}
In particular $\mathbb{E}^{\omega}(\langle\delta_{n},E_{H^{\omega}_\Lambda}(\{r\})\delta_{n}\rangle)=0$ for $r\in \mathbb{R}$. So
\begin{align*}
\mathbb{E}^{\omega}(\langle\delta_n,E_{H^\omega_\Lambda}(I)\delta_n\rangle )\leq   \pi(1 +\lVert n\rVert_2^{\alpha})^{-1} |I|,
~~n\in \Lambda\subseteq \mathbb{Z}^d.
\end{align*}
\end{proof}
\noindent In the next proposition we prove the absolute continuity of $N_j$ inside $(a_j, a_{j+1})$ (see \cite[Lemma 5.2]{GJMS}). 
Regularity of $N_j$ are required for the statistics because we do not have the optimal Wegner estimate inside $(a_j, a_{j-1})$ for $j\ge 2$.
\begin{proposition}\label{pro2}
For $\frac{d}{k}>\alpha>\frac{d}{k+1}$ and $1\leq j\leq k$, there exists $C_j>0$ such that 
$$N_j (b)-N_j(a)\leq C_j|b-a|,~\forall~(a,b)\subset (a_{j+1},a_j).$$
\end{proposition}
\begin{proof}
The result follows from \cite[Lemma 5.2]{GJMS} and the fact that the operator $H^{(i)}$ (following the notations as in \cite{GJMS}) is of the form
$$A(\lambda_1,\cdots,\lambda_l)=H_0+\sum_{i=1}^l \lambda_i |\delta_i\rangle\langle\delta_i|~~~l\ge j,$$
and so the question reduces to computing $|n|^{j\alpha}\mathbb{E}(\chi_n(\omega))$ where
$$\chi_n(\omega)=\left\{\begin{matrix}N_{(a,b)}(H^{(i)}), & if~n\in Z^{(i)}~and~d(Z^{(i)},0)=|n|\\ 0 & otherwise\end{matrix}\right..$$
Here for the definition of $Z^{(i)}$, we are taking 
$$A^\omega=\{n:V^\omega(n)\leq \min\{a_{j-1},|n|^\gamma\}\},$$
and similarly the event $B_n$ is defined as
$$B_n=\{\omega: \text{there exist $\{n_i\}_{i=1}^{k+1}$ in $B(n,|n|^\gamma)$ such that }V^\omega(n_i)\leq min\{a_{j-1},|n|^\gamma\}\},$$
for any $0<\gamma\ll 1$.

Now using the fact that $\chi_n(\cdot)$ is non-zero only on a set of configuration whose probability is of the order of $n^{-j\alpha}$ we end up with
$$|n|^{j\alpha}\mathbb{E}^\omega[\chi_n(\omega)]=|n|^{j\alpha}\mathbb{P}[\chi_n(\omega)\neq 0]\mathbb{E}^\omega[\chi_n(\omega)|\chi_n(\omega)\neq 0],$$
so the absolute continuity follows the same steps as Proposition \ref{pro1}, when it is applied to $H^{(i)}$.

\end{proof}
\noindent Combes-Germinet-Klein \cite{JFA} (see also Combes-Hislop-Klopp\cite{CHK})
estimated the Wegner and Minami estimates for generalized Anderson Model. The following proposition give the Wegner and Minami estimate for the Model (\ref{Model}).
\begin{proposition}
 \label{minami}
 For $ \Lambda \subset \mathbb{Z}^d$ finite volume and bounded interval $I\subset\mathbb{R}$ we have
 \begin{equation}
  \label{wegner}
\mathbb{E}\big(Tr E_{H^\omega_\Lambda}(I)\big)\leq \pi\sum_{n\in \Lambda}b_n^{-1}|I|
 \end{equation}
\begin{equation}
 \label{Minami}
\mathbb{E}\bigg(\big(Tr E_{H^\omega_\Lambda}(I)\big)\big(Tr E_{H^\omega_\Lambda}(I)-1\big)\bigg)\leq \bigg(\pi\sum_{n\in \Lambda}b_n^{-1}|I|\bigg)^2.
\end{equation}
\end{proposition}
\noindent The Wegner estimate \eqref{wegner} follows from Proposition \ref{pro1}.
The proof of \eqref{Minami} goes exactly the same way as in \cite[Theorem 2.1]{JFA} once we have \eqref{wegner}, hence we are omit the proof here.

One important thing to notice here is that the above estimates are independent of position of $I$. 
We will need a modified version of above estimates where position of $I$ and the distance of the box from the origin are also taken in account.
\begin{proposition}\label{modEst}
Let $l\gg 1$ and let $\Lambda\subset\ZZ^d$ be such that $dist(0,\Lambda)>l$. For $2\leq j\leq k$ let $I\subset (a_j,a_{j-1})$ be a bounded interval, then
\begin{equation}\label{modEstEq1}
\mathbb{E}\left[ Tr E_{H^\omega_\Lambda}(I)\right]\leq M_{j}\frac{|\Lambda|}{l^{j\alpha}}|\Lambda||I|,
\end{equation}
\begin{equation}\label{modEstEq2}
\mathbb{P}\left[Tr E_{H^\omega_\Lambda}(I)\big)\geq 2\right]\leq \tilde{M}_j\frac{|\Lambda|}{l^{j\alpha}} (|\Lambda||I|)^2.
\end{equation}
\end{proposition}
\begin{proof}
Since $I\subset (a_j,a_{j-1})$, for $H^\omega_\Lambda$ to have an eigenvalue in the interval $I$, we need a connected subset $G\subset\Lambda$ of size at least $j$ where the random variables are small. 
This follows from the fact that $H^\omega_\Lambda$ is positive operator so any restriction by a projection is only smaller than the full operator.
Hence
\begin{align*}
\mathbb{P}[\sigma(H^\omega_\Lambda)\cap I\neq \phi]\leq \hat{M}_j{|\Lambda| \choose 1}\frac{1}{l^{j\alpha}},
\end{align*}
where the constant $\hat{M}_j$ is combinatorial in nature and depends only on $j$. 
So to prove \eqref{modEstEq1} or \eqref{modEstEq2}, we compute the Wegner and Minami estimate under the condition that on the set $G$, the random variables are bounded by $l^{-\alpha}$.

\end{proof}
\begin{remark}
There are few important things that should be noted:
\begin{enumerate}
\item First note that the estimates \eqref{modEstEq1} or \eqref{modEstEq2} are weaker than \eqref{wegner} and \eqref{Minami} in terms of $|\Lambda|$.
But on other hand we are taking in account of the location of $\Lambda$ with respect to origin, 
which helps in making the expression smaller than usual Wegner or Minami estimates, if $|\Lambda|$ is  relatively small with respect to $l$, the distance of $\Lambda$ from origin.
\item The results will be used over cubes of the form $\Lambda_{l_L}(p)=\{x\in\ZZ^d: \|x-p\|<l_L\}$,
where $l_L\approx L^\mu$ and $\|p\|>L^\theta$ where $\theta$ is close to one and $\mu$ is small enough.
This will be used for \cite[Theorem 2.1]{Klopp} in place of Wegner and Minami estimates.
\end{enumerate}

\end{remark}
\noindent Following theorem is restatement of \cite[Theorem 2.1]{Klopp} with some modifications. 
One of the important modification is, instead of the operator $H^\omega_{\Lambda_L}$ we will be working with the operator
$$\tilde{H}^\omega_L=(I-\chi_{\Lambda_{m_L}})H^\omega_{\Lambda_L}(I-\chi_{\Lambda_{m_L}}),$$
on the space $\ell^2(\Lambda_L\setminus \Lambda_{m_L})$ where $m_L\approx L^\theta$ for some appropriately chosen $\theta\in(0,1)$. 
The advantage gained is that we can use the results of Proposition \ref{modEst} instead of usual Wegner and Minami estimates on the smaller cubes.
The other modification is that we have already mentioned all the scales (the parameters $\theta,\beta,\gamma$) involved in the problem.
\begin{theorem}\label{infDivThm}
For $2\leq j\leq k$, and the parameters $0<\theta,\beta,\gamma<1$ satisfying
$$0<\kappa:=\frac{j\alpha}{d-j\alpha}(1-\theta)<\frac{1}{6(d+1)},$$
$$\frac{(2d+1)\kappa+1}{2}<\beta<1-\kappa,$$
$$0<\gamma<\frac{2\beta-(2d+1)\kappa-1}{2d\beta+(2d+1)\kappa+1},$$
set $\mu$ to be
\begin{align*}
(d-j\alpha)\left(\frac{\gamma\beta}{1+\gamma}+\kappa\right)<\mu<\min\left\{1,\frac{(d-j\alpha)}{2d}\left(\frac{2\beta}{1+\gamma}-1-\kappa\right),\frac{d-j\alpha}{d}(1-\beta)\right\},
\end{align*}
and also set $R>d$.
Consider a sequence of intervals $\{I_L\}_L$ in $(a_j,a_{j-1})$ such that
$$N_j(I_L)\approx L^{-(d-j\alpha)\beta},~N_j(I_L)\geq |I_L|^{1+\gamma},$$
and set
$$m_L=\lfloor L^\theta\rfloor,~l_L=\lfloor L^\mu\rfloor,~l_L^\prime=\lfloor (R\log L^d)^{2}\rfloor.$$
For any $p>0$ and sufficiently large $L$ (depending on $\kappa,\beta,\gamma,\mu,p$), there exists:
\begin{itemize}
 \item A decomposition of $\Lambda_L\setminus\Lambda_{m_L}$ into disjoint cubes of the form $\Lambda_{l_L}(p_j)=\{x\in\ZZ^d: \|x-p\|<l_L\}$, such that
 \begin{itemize}
  \item $\cup_i\Lambda_{l_L}(p_i)\subset \Lambda_L\setminus\Lambda_{m_L}$,
  \item  $dist(\Lambda_{l_L}(p_i),\Lambda_{l_L}(p_{\tilde{i}}))\geq l_L^\prime$ for $i\neq \tilde{i}$,
  \item $dist(\Lambda_{l_L}(p_i),\partial (\Lambda_L\setminus\Lambda_{m_L}))\geq l_L^\prime$, here
  \item $\#[(\Lambda_L\setminus\Lambda_{m_L})\setminus (\cup_i\Lambda_{l_L}(p_i))]\leq O(L^d\frac{l_L^\prime}{l_L})$.
 \end{itemize}

\item A set of configuration $\mathcal{Z}_L\subset\Omega$ such that
\begin{equation}\label{infDivThmEq1}
 \mathbb{P}(\mathcal{Z}_L)\geq 1- L^{-pd}-\exp\left(-c_1 L^{\eta_1}\right)-\exp\left(-c_2 L^{\eta_2}\right)
\end{equation}
where $c_1,c_2>0$ and 
$$0<\eta_1<(d-j\alpha)(1-\beta)-d\mu,$$
$$\max\left\{(d-j\alpha)\left(1-\frac{\beta}{1+\gamma}+\kappa\right)-\mu,0\right\}<\eta_2^\prime<\eta_2<(d-j\alpha)\left(1-\beta\right).$$

\item For $\omega\in\mathcal{Z}_L$, there exists at least $\frac{L^d}{l_L^d}\left(1+O(L^{\eta_1-d(1-\mu)})\right)$ disjoint boxes $\{\Lambda_{l_L}(p_i)\}_i$ such that
\begin{itemize}
 \item The operator $H^\omega_{\Lambda_{l_L}(p_i)}$ has at most one eigenvalue $E_{i,n}^\omega$ in $I_L$,
 \item The box $\Lambda_{l_L}(p_i)$ contains at most one center of localization of, say $x_{k_i}^{\omega,L}$ of an eigenvalue of $\tilde{H}^\omega_L$ in $I_L$, say $E_{k_i,L}^\omega$,
 \item The box $\Lambda_{l_L}(p_i)$ contains a center of localization $x_{k_i}^{\omega,L}$ if and only if $\sigma(H^\omega_{\Lambda_{l_L}(p_i)})\cap I_L\neq \phi$, in which case, one has
 \begin{equation}\label{infDivThmEq2}
  |E_{i,n}^\omega-E_{k_i,L}^\omega|<L^{-R},
 \end{equation}
 and
 \begin{equation}\label{infDivThmEq3}
  dist(x_{k_i}^{\omega,L},(\Lambda_L\setminus\Lambda_{m_L})\setminus \Lambda_{l_L}(p_i))>l_L^\prime.
 \end{equation}

\end{itemize}

\item The number of eigenvalues of $\tilde{H}^\omega_L$ that are not described above is bounded by:
\begin{equation}\label{infDivThmEq4}
 O\left( (L^{d-j\alpha}N_j(I_L))\left[L^{\eta_1+d\mu-(d-j\alpha)(1-\beta)}+L^{\eta^\prime_2-(d-j\alpha)(1-\beta)} \right]\right),
\end{equation}

and by the choices of parameters it is $o(L^{d-j\alpha}N_j(I_L))$ as $L\rightarrow\infty$.
\end{itemize}

\end{theorem}
\begin{proof}
The proof goes exactly same as the proof of \cite[Theorem 2.1]{Klopp} (see also \cite[Theorem 1.2]{GK1}), so we will skip most of it here and focus on important details.
The main difference is that instead of usual Wegner and Minami estimates, the equations \eqref{modEstEq1} and \eqref{modEstEq2} are used.

One of the important step is to estimate the probability that the number of cubes such that $H^\omega_{\Lambda_{l_L}(p)}$ has more than one eigenvalue is greater than 
$K_L\approx L^{\eta_1}$ for some $\eta_1>0$. For large enough $L$ this is given by
\begin{align*}
&\leq \sum_{k>K_L} {|P_L| \choose k } \left(\tilde{M}_j\frac{l_L^{3d}}{m_L^{j\alpha}}|I_L|^2\right)^k\approx \left(e\tilde{M}_j \frac{L^d l_L^{2d} |I_L|^2}{K_L  m_L^{j\alpha}}\right)^{K_L}
\end{align*}
if 
$$\frac{L^d l_L^{2d} |I_L|^2}{m_L^{j\alpha}}<1.$$
Here we denote the set $P_L$ to be the centers of cube which decomposes $\Lambda_L\setminus\Lambda_{m_L}$ as in the statement of theorem. 
Above condition is satisfied if 
\begin{equation}\label{infDivThmPfEq3}
(d-j\alpha)\left( 1-\frac{2\beta}{1+\gamma}+\kappa\right)+2d\mu<0.
\end{equation}
We also need $\frac{K_Ll_L^d}{N(I_L)L^{d-j\alpha}}$ to converge to zero, for this we take
\begin{equation}\label{infDivThmPfEq1}
 \mu<\frac{(d-j\alpha)}{d}(1-\beta).
\end{equation}
Other part is to estimate the number of eigenvalues in $I_L$ for which the center of localization lies outside $\cup_i \Lambda_{l_L-l^\prime_L}(p_i)$.
To get an probability estimate of eigenvalues in $I_L$ for which the center of localization lies outside $\cup_i \Lambda_{l_L-l^\prime_L}(p_i)$ is less than $K_L^\prime\approx L^{\eta_2}$, we 
need
$$(d-j\alpha)\left(1-\frac{\beta}{1+\gamma}+\kappa\right)-\mu<\eta_2$$
and we also require  
$$\eta_2<(d-j\alpha)(1-\beta),$$
so we end up with
\begin{equation}\label{infDivThmPfEq2}
 (d-j\alpha)\left(\frac{\beta\gamma}{1+\gamma}+\kappa\right)-\mu<0
\end{equation}
Other than that we need to make sure that the interval for the choice of $\mu$ is non-empty and intersects $[0,1]$ non-trivially.

\end{proof}

\section{Structure of $N_1$}
In this section we will give the proof of Theorem \ref{thmReg1}. In the large disorder
case, the kinetic energy part of the random Hamiltonian can be viewed as perturbation of diagonal randomness.
In such case one can use a random walk expansion, in which the expectation of each of the terms turns out to
have analytic continuation through part of real axis (see Krishna-Kaminaga-Nakamura \cite{KKN}). In case of growing randomness a similar idea works and the
local density of states has an analytic continuation.\\\\
For $0<\theta<1-\frac{\alpha}{d}$ set $m_L=\lfloor L^\theta\rfloor$ and consider the operator
$$H^\omega_{L,m_L}=\chi_{\Lambda_{m_L}}H^\omega_{\Lambda_L}\chi_{\Lambda_{m_L}}+(I-\chi_{\Lambda_{m_L}})H^\omega_{\Lambda_L}(I-\chi_{\Lambda_{m_L}}).$$
For some $M>2d$ define
$$G_L(z)=\frac{1}{L^{d-\alpha}}\sum_{n\in\Lambda_L\setminus \Lambda_{m_L}}\mathbb{E}^\omega\left [\left\langle\delta_n,(H^\omega_{L,m_L}-z)^{-1}\delta_n\right\rangle- \left\langle\delta_n,(H^\omega_{L,m_L}-\iota M)^{-1}\delta_n\right\rangle\right].$$
Here $\chi_{\Lambda_{l}}$ denotes the canonical projection onto the subspace $\ell^2(\Lambda_l)$.
The reason to focus on this function is because of following lemma.
\begin{lemma}\label{lem1reg}
For $z\in\mathbb{C}^{+}$ we have
$$\left|\frac{1}{L^{d-\alpha}}\mathbb{E}^\omega\left[tr\left((H^\omega_{\Lambda_L}-z)^{-1}-(H^\omega_{\Lambda_L}-\iota M)^{-1}\right)\right]-G_L(z)\right|\leq \frac{6}{L^{d-\alpha-d\theta}(Im z)^2} .$$
\end{lemma}
\begin{proof}
By resolvent equation we have
\begin{align*}
& (H^\omega_{\Lambda_L}-z)^{-1}-(H^\omega_{L,m_L}-z)^{-1}\\
&\qquad=(H^\omega_{\Lambda_L}-z)^{-1}[\chi_{\Lambda_{m_L}}\Delta (I-\chi_{\Lambda_{m_L}})+(I-\chi_{\Lambda_{m_L}})\Delta \chi_{\Lambda_{m_L}}](H^\omega_{L,m_L}-z)^{-1}\\
\Rightarrow\qquad & tr((H^\omega_{\Lambda_L}-z)^{-1})-tr((H^\omega_{L,m_L}-z)^{-1})\\
&\qquad= tr((H^\omega_{L,m_L}-z)^{-1} (H^\omega_{\Lambda_L}-z)^{-1}[\chi_{\Lambda_{m_L}}\Delta (I-\chi_{\Lambda_{m_L}})+(I-\chi_{\Lambda_{m_L}})\Delta \chi_{\Lambda_{m_L}}]).
\end{align*}
Hence we have
\begin{align*}
& \frac{1}{L^{d-\alpha}}tr((H^\omega_{\Lambda_L}-z)^{-1})-\frac{1}{L^{d-\alpha}}\sum_{n\in\Lambda_L\setminus \Lambda_{m_L}}\langle \delta_n,(H^\omega_{L,m_L}-z)^{-1}\delta_n\rangle\\
&\qquad\qquad = \frac{1}{L^{d-\alpha}}\sum_{n\in\Lambda_{m_L}}\langle \delta_n,(H^\omega_{L,m_L}-z)^{-1}\delta_n\rangle\\
&\qquad\qquad\qquad +\frac{1}{L^{d-\alpha}}\sum_{\substack{(n,m)\in \Lambda_{m_L}\times (\Lambda_L \setminus\Lambda_{m_L})\\ \|n-m\|_1=1  }}  \langle \delta_n (H^\omega_{L,m_L}-z)^{-1} (H^\omega_{\Lambda_L}-z)^{-1} \delta_m \rangle \\
&\qquad\qquad\qquad +\frac{1}{L^{d-\alpha}}\sum_{\substack{(n,m)\in \Lambda_{m_L}\times (\Lambda_L \setminus\Lambda_{m_L})\\ \|n-m\|_1=1  }}  \langle \delta_m (H^\omega_{L,m_L}-z)^{-1} (H^\omega_{\Lambda_L}-z)^{-1} \delta_n\rangle.
\end{align*}
So taking expectation we have
\begin{align*}
\left|\frac{1}{L^{d-\alpha}}\mathbb{E}^\omega\left[tr\left((H^\omega_{\Lambda_L}-z)^{-1}-(H^\omega_{\Lambda_L}-\iota M)^{-1}\right)\right]-G_L(z)\right|\leq \frac{6 m_L^d}{L^{d-\alpha}(Im z)^2},
\end{align*}
which completes the proof.

\end{proof}
\noindent So this lemma shows that the limit of Borel transform of $\nu_L$ coincides with limit of $G_L$ for any $z\in\mathbb{C}^{+}$ and the convergence is uniform on every compact set of $\mathbb{C}^{+}$. 
Hence all we have to do is find the limit of the sequence $G_L$. 
\begin{lemma}\label{lem2reg}
For any compact subset  $U\subset\{z\in\mathbb{C}^{+}: Re z>2M^2\}$, there exists $L_0\in\mathbb{N}$ such that for $L>L_0$ the function $G_L$ has an analytic continuation to whole of $U$, and
$$\sup_{z\in U}\left|G_L(z)+ C_{d,\alpha}\ln\frac{2d-z}{2d-\iota M} \right|\xrightarrow{L\rightarrow \infty} 0,$$
where 
$$C_{d,\alpha}=\lim_{L\rightarrow\infty}\frac{1}{L^{d-\alpha}}\sum_{n\in\Lambda_L}\frac{1}{b_n}.$$
\end{lemma}
\begin{proof}
Let $\delta>0$, set $L_0$ to be large enough so that 
$$\frac{1}{b_m}\sup_{z\in U}\left| \ln\frac{b_m+2d-z}{2d-z}\right|<\frac{1}{b_m^\delta M}\qquad\forall \|m\|_\infty>L_0,$$
and 
$$\inf_{z\in U}|b_m+2d-z|>2M^2\qquad\forall \|m\|_\infty>L_0.$$
Let $U_0\subset \mathbb{C}$ be an open subset such that above two conditions hold. 
It is clear that $U\subset U_0$ and by choosing $L_0$ larger we can also make sure that $U_0\cap \{z\in \mathbb{C}: Im z>2d\}$ is non-empty.

Now using random walk expansion of Green's function for $Im z>2d$ for $n\in\Lambda_L\setminus\Lambda_{m_L}$, we have
\begin{align*}
& \langle\delta_n, (H^\omega_{L,m_L}-z)^{-1}\delta_n\rangle-\langle\delta_n, (V(\omega)+2d-z)^{-1}\delta_n\rangle\\
&\qquad =\left\langle\delta_n, \left(V(\omega)+2d-\Delta_{\Lambda_L\setminus\Lambda_{m_L}}-z\right)^{-1}\delta_n\right\rangle-\left\langle\delta_n, (V(\omega)+2d-z)^{-1}\delta_n\right\rangle\\
&\qquad =\left\langle\delta_n, (V(\omega)+2d-z)^{-1}\left(I-\Delta_{\Lambda_L\setminus\Lambda_{m_L}}(V(\omega)+2d-z)^{-1}\right)^{-1}\delta_n\right\rangle\\
&\qquad\qquad -\left\langle\delta_n, (V(\omega)+2d-z)^{-1}\delta_n\right\rangle\\
&\qquad=\sum_{k=1}^\infty \left\langle\delta_n, (V(\omega)+2d-z)^{-1}\left(\Delta_{\Lambda_L\setminus\Lambda_{m_L}}(V(\omega)+2d-z)^{-1}\right)^{k}\delta_n\right\rangle\\
&\qquad=\sum_{k=1}^\infty \sum_{\gamma\in \Gamma^k_{L,m_L}(n)} \prod_{i=0}^k \frac{1}{b_{\gamma_i}q_{\gamma_i}(\omega)+2d-z},
\end{align*}
where 
$$\Gamma^k_{L,m_L}(n)=\{\gamma:\{0,\cdots,k\}\rightarrow (\Lambda_L\setminus \Lambda_{m_L}): \gamma_0=\gamma_k=n, \|\gamma_i-\gamma_{i-1}\|_1=1~\forall 1\leq i\leq k\}.$$
For $\gamma\in\Gamma^k_{L,m_L}(n)$ denote the set $[\gamma]=\{\gamma_i: 0\leq i\leq k\}$, and for $m\in[\gamma]$ set $\#(m|\gamma)=\#\{i:\gamma_i=m\}$, then
\begin{align}\label{lem2regeq1}
& \langle\delta_n, (H^\omega_{L,m_L}-z)^{-1}\delta_n\rangle-\langle\delta_n, (V(\omega)+2d-z)^{-1}\delta_n\rangle\nonumber\\
&\qquad=\sum_{k=1}^\infty \sum_{\gamma\in \Gamma^k_{L,m_L}(n)} \prod_{m\in [\gamma]} \frac{1}{(b_{m}q_{m}(\omega)+2d-z)^{\#(m|\gamma)}} .
\end{align}
For $z\in U_0$ the expression
$$\sum_{\gamma\in \Gamma^k_{L,m_L}(n)} \prod_{m\in [\gamma]} \int_0^1 \frac{1}{(b_{m} \omega_m+2d-z)^{\#(m|\gamma)}}d\omega_m$$
can be computed through contour integration and get
\begin{align}\label{lem2regeq2}
&\sum_{\gamma\in \Gamma^k_{L,m_L}(n)} \prod_{m\in [\gamma]} \int_0^1 \frac{1}{(b_{m} \omega_m+2d-z)^{\#(m|\gamma)}}d\omega_m\\
&\qquad= \sum_{\gamma\in \Gamma^k_{L,m_L}(n)} \frac{1}{\prod_{m\in[\gamma]}b_m} \left(\prod_{m: \#(m|\gamma)=1} \ln \frac{b_m+2d-z}{2d-z}\right)\nonumber\\
&\qquad\qquad\qquad\left[\prod_{m:\#(m|\gamma)>1}\frac{1}{\#(m|\gamma)-1}\left(\frac{1}{(b_m+2d-z)^{\#(m|\gamma)-1}}- \frac{1}{(2d-z)^{\#(m|\gamma)-1}} \right)\right].\nonumber
\end{align}
Hence using the definition of $L_0$, and using the observation $|\Gamma^k_{L,m_L}(n)|\leq (2d)^k$ and $\#(n|\gamma)\geq 2$, for any  $L>L_0$ we have
\begin{align}\label{lem2regeq3}
&\left|\sum_{\gamma\in \Gamma^k_{L,m_L}(n)} \prod_{m\in [\gamma]} \int_0^1 \frac{1}{(b_{m} \omega_m+2d-z)^{\#(m|\gamma)}}d\omega_m\right|\nonumber\\
&\qquad\leq \frac{1}{b_n \min_{|n-m|_1=1}b_m^\delta} \frac{(2d)^k}{M^k}.
\end{align}
So for $z\in U_0\cap \{z: Im z>2d\}$ we can take the expectation of equation \eqref{lem2regeq1} and get
\begin{align*}
&\mathbb{E}^\omega\left[\langle\delta_n, (H^\omega_{L,m_L}-z)^{-1}\delta_n\rangle\right]-\mathbb{E}\left[\langle\delta_n, (V(\omega)+2d-z)^{-1}\delta_n\rangle\right]\\
&\qquad =\sum_{k=1}^\infty \sum_{\gamma\in \Gamma^k_{L,m_L}(n)} \prod_{m\in [\gamma]} \int_0^1 \frac{1}{(b_{m}\omega_m+2d-z)^{\#(m|\gamma)}}d\omega_m.
\end{align*}
By \eqref{lem2regeq2} and \eqref{lem2regeq3}, we get that 
$$\mathbb{E}^\omega\left[\langle\delta_n, (H^\omega_{L,m_L}-z)^{-1}\delta_n\rangle\right]-\mathbb{E}\left[\langle\delta_n, (V(\omega)+2d-z)^{-1}\delta_n\rangle\right]$$
has an analytic continuation up to $U_0$ and using \eqref{lem2regeq3} we have
$$|\mathbb{E}^\omega\left[\langle\delta_n, (H^\omega_{L,m_L}-z)^{-1}\delta_n\rangle\right]-\mathbb{E}\left[\langle\delta_n, (V(\omega)+2d-z)^{-1}\delta_n\rangle\right]|
\leq \frac{\frac{2d}{M-2d} }{b_n \min_{|n-m|_1=1}b_m^\delta} .$$
Hence we have
\begin{align}
\label{ana-den}
\left|\frac{1}{L^{d-\alpha}} \sum_{n\in\Lambda_L\setminus \Lambda_{m_L}}\left (\mathbb{E}^\omega 
\left[ \langle\delta_n, (H^\omega_{L,m_L}-z)^{-1}\delta_n\rangle \right]- \int_0^1 \frac{d\omega_n}{b_n \omega_n +2d-z}\right) \right|\nonumber\\
\leq \frac{\frac{2d}{M-2d} }{L^{d-\alpha}}\sum_{n\in\Lambda_L\setminus \Lambda_{m_L}} \frac{1}{{b_n \min_{|n-m|_1=1}b_m^\delta}}.
\end{align}
Using above we get
\begin{align}\label{lem2regeq4}
\left| G_L(z)-\frac{1}{L^{d-\alpha}} \sum_{n\in\Lambda_L\setminus \Lambda_{m_L}} \left(\int_0^1 \frac{d\omega_n}{b_n \omega_n +2d-z}-\int_0^1 \frac{d\omega_n}{b_n \omega_n +2d-\iota M}\right) \right|\nonumber\\
\leq \frac{2 \frac{2d}{M-2d} }{L^{d-\alpha}}\sum_{n\in\Lambda_L\setminus \Lambda_{m_L}} \frac{1}{{b_n \min_{|n-m|_1=1}b_m^\delta}},
\end{align}
but then we have
\begin{align*}
&\frac{1}{L^{d-\alpha}} \sum_{n\in\Lambda_L\setminus \Lambda_{m_L}} \left(\int_0^1 \frac{d\omega_n}{b_n \omega_n +2d-z}-\int_0^1 \frac{d\omega_n}{b_n \omega_n +2d-\iota M}\right)\\
&=\frac{1}{L^{d-\alpha}} \sum_{n\in\Lambda_L\setminus \Lambda_{m_L}} \frac{1}{b_n}\left(\ln \frac{b_n+2d-z}{b_n+2d-\iota M}-\ln \frac{2d-z}{2d-\iota M}\right)\\
&\xrightarrow{L\rightarrow\infty} -C_{d,\alpha} \ln\frac{2d-z}{2d-\iota M}.
\end{align*}
Here we are using the fact that given a positive sequence $\{c_n\}_n$ and $\{d_n\}_n$ if
$$\sum_{n=1}^N c_n\xrightarrow{N\rightarrow\infty} \infty,~\&~~~d_n\xrightarrow{n\rightarrow\infty} 0,$$
then
$$ \frac{1}{\sum_{n=1}^N c_n}\sum_{n=1}^N c_n d_n\xrightarrow{N\rightarrow\infty}0.$$ 
Using the above relation for \eqref{lem2regeq4} we conclude 
$$\left|G_L(z)+C_{d,\alpha} \ln\frac{2d-z}{2d-\iota M} \right|\xrightarrow{L\rightarrow\infty}0,$$
which completes the proof of the lemma.
\end{proof}
\begin{remark}
 \label{ana-cont}
Defining 
$$N_{L,m_L}(E)=\frac{1}{L^{d-\alpha}} \mathbb{E}^\omega[tr(E_{(I-\chi_{\Lambda_{m_L}}) H^\omega_L (I-\chi_{\Lambda_{m_L}})}(a_1,E))]$$
for $E>a_1$, observe that $G_L(\cdot)$ is the Borel transform of the measure associated with the distribution function $N_{L,m_L}$.
Above proof shows that $G_L$ is analytic for any compact $U\subset \{z\in\mathbb{C}: \Re z>2M^2\}$ for large enough $L$.
So for any compact subset $V\subset(2M^2,\infty)$, the density of the measure associated with the distribution function $N_{L,m_L}$ is analytic for large enough $L$.
In the above proof we have shown the compact uniform convergence of $G_L$, so we have compact uniform convergence of $\{\frac{d N_{L,m_L}}{dE}(\cdot)\}_L$ on $(2M^2,\infty)$.
Now using the observation 
$$\left|\nu_L(a_1,E)-N_{L,m_L}(E)\right|\leq O\left(\frac{m_L^d}{L^{d-\alpha}}\right),$$
we also have $N_{L,m_L}(E)\rightarrow N_1(E)$ for each $E>0$. We have
\begin{equation}
\label{uni-den}
N^{\prime}_{L,m_L}(E)\xrightarrow[compact~uniform]{L\to\infty} N_1^{\prime}(E)~on~(2M^2,\infty).
\end{equation}

\end{remark}

\noindent Combining both the above lemmas we have the explicit expression for $N_1(x)=\nu(2d,x)$.
\subsubsection*{Proof of Theorem \ref{thmReg1}}
For any compact subset $U\subset \mathbb{C}^{+}$, combining Lemma \ref{lem1reg} and Lemma \ref{lem2reg} we conclude that 
$$\lim_{L\rightarrow \infty} \sup_{z\in U}\left| \frac{1}{L^{d-\alpha}} \mathbb{E}^\omega\left[tr\left( (H^\omega_{\Lambda_L}-z)^{-1}-(H^\omega_{\Lambda_L}-\iota M)^{-1} \right)\right]+C_{d,\alpha}\ln \frac{2d-z}{2d-\iota M}\right|=0.$$
So using the fact that
$$\frac{1}{L^{d-\alpha}} \mathbb{E}^\omega\left[tr\left( (H^\omega_{\Lambda_L}-z)^{-1}-(H^\omega_{\Lambda_L}-\iota M)^{-1} \right)\right]=\int \left(\frac{1}{x-z}-\frac{1}{x-\iota M}\right)\nu_L(dx)$$
and $\nu_L$ converges to $\nu$ in distribution we have
$$\int \left(\frac{1}{x-z}-\frac{1}{x-\iota M}\right)\nu(dx)= -C_{d,\alpha}\ln \frac{2d-z}{2d-\iota M}$$
for $z\in\mathbb{C}^{+}$. Since $\nu$ has support in $(0,\infty)$, we have uniqueness of Borel transform which provides
$$\nu(2d,x)=C_{d,\alpha}(x-2d)\qquad\forall x>2d,$$
completing the proof.

\qed

\section{Proof of main results:}
In this section we will give the proof of Theorem \ref{thm1} and Theorem \ref{main}.
\subsubsection*{ Proof of Theorem \ref{thm1}:}
The proof is divided into simpler steps. 
First we will show that the point process can be approximated with another point process associated with $\tilde{H}^\omega_L$ (defined below).
Then we will divide the interval $(a_j,a_{j-1})$ into smaller intervals and show that one can ignore some of these intervals and for rest of the intervals, 
we can show $N_j(I_{L,i})\geq |I_{L,i}|^{1+\gamma}$ and we know the length of the intervals.
Finally we will show the infinite divisibility of the point process on these interval with high probability through Theorem \ref{infDivThm}, and compute the limit.
\\\\
\noindent{\bf Step 1:}
We will approximate the point process $\{\xi^\omega_{L,j,t}\}_L$ by another point process $\{\Xi^\omega_{t,(a_j,a_{j-1}),L}\}_L$.
For the new point process, set $m_L=\lfloor L^\theta\rfloor$ where  $0<\theta<1$, and consider the operator
$$\tilde{H}^\omega_{L}=(I-\chi_{\Lambda_{m_L}})H^\omega_{\Lambda_L}(I-\chi_{\Lambda_{m_L}}),$$
i.e $H^\omega_{\Lambda_L}$ restricted onto the subspace $\ell^2(\Lambda_L\setminus\Lambda_{m_L})$.
For an interval $I=(a,b)$ define the point process
$$\Xi^\omega_{t,I,L}(\cdot)=\sum_{E\in \sigma(\tilde{H}^\omega_L)\cap I}\delta_{N_j(I)L^{d-j\alpha}(N_{j,I}(E)-t)}$$
where $N_j(I)=N_j(b)-N_j(a)$ and $N_{j,I}(x)=\frac{N_j(x)-N_j(a)}{N_j(I)}$.

The authors in \cite{GJMS}  used multiscale analysis to show the spectral localization of $H^\omega$, so the work of
Germinet-Klein \cite{GA} gives the dynamical localization for the model.
Using \cite[Lemma 1.1]{GK1} we conclude that for $p>d$ and $\frac{1}{2}<\varrho<1$ there exists a set of configurations $\Omega_L$ such that 
$$\mathbb{P}[\Omega_L]>1-L^{-p}$$
and for any $\omega\in\Omega_L$ the normalized eigenfunction $\psi_{L,i}^\omega$ for an eigenvalue $E_{L,i}^\omega\in\sigma(H^\omega_{\Lambda_L})\cap I$ satisfies
$$|\psi_{L,i}^\omega(x)|\leq L^{p+d}e^{-|x-x_{L,i}^\omega|^\varrho}\qquad\forall x\in\Lambda_L,$$
where the center of localization $x_{L,i}^\omega\in\Lambda_L$ is a point where $x\mapsto |\psi_{L,i}^\omega(x)|$ attain its maximum. 
So defining $\tilde{m}_L=\lfloor m_L^a\rfloor$ for some $0<a<1$, 
let $E_{L,i}^\omega\in\sigma(H^\omega_{\Lambda_L})\cap I$ be such that the center of localization lies in $\Lambda_L\setminus \Lambda_{m_L+\tilde{m}_L}$, we have
$$\frac{\|(\tilde{H}^\omega_L-E^\omega_{L,i})\psi^\omega_{L,i}\|_2^2}{\|(I-\chi_{\Lambda_{m_L}})\psi^\omega_{L,i}\|_2^2}<e^{-L^{a\theta\varrho^\prime}}$$
where $0<\varrho^\prime<\varrho$, for $L$ large enough. 
So there exists an eigenvalue $\tilde{E}^\omega_{L,i}\in\sigma(\tilde{H}^\omega_L)$ such that 
\begin{equation}\label{appEq1}
 |\tilde{E}^\omega_{L,i}-E^\omega_{L,i}|< e^{-L^{a\theta\varrho^\prime}}.
\end{equation}
Using the same argument we can show that if $\tilde{E}_{L,i}^\omega\in\sigma(\tilde{H}^\omega_{L})\cap I$ 
with center of localization $\tilde{x}_{L,i}^\omega$ in $\Lambda_L\setminus \Lambda_{m_L+\tilde{m}_L}$, then $H^\omega_{\Lambda_L}$ has an eigenvalue satisfying \eqref{appEq1}.
Finally since the length of the interval where the eigenvalue can lie is exponential in nature, using Minami estimate and above, the set of configuration
\begin{align*}
\Omega_L&=\{\omega\in\Omega: \forall E_{L,i}^\omega\in\sigma(H^\omega_{\Lambda_L})\cap I~s.t~ x_{L,i}^\omega\in\Lambda_L\in \Lambda_L\setminus \Lambda_{m_L+\tilde{m}_L}\\
&\qquad\qquad\exists~unique~\tilde{E}_{L,i}^\omega\in\sigma(\tilde{H}^\omega_{L})\cap I,~s.t~|E_{L,i}^\omega-\tilde{E}_{L,i}^\omega|<e^{-L^{a\theta\varrho^\prime}},~and\\
&\qquad\qquad \forall \tilde{E}_{L,j}^\omega\in\sigma(\tilde{H}^\omega_{L})\cap I~s.t~ \tilde{x}_{L,j}^\omega\in\Lambda_L\in \Lambda_L\setminus \Lambda_{m_L+\tilde{m}_L}\\
&\qquad\qquad\exists~unique~E_{L,j}^\omega\in\sigma(H^\omega_{\Lambda_L})\cap I,~s.t~|E_{L,j}^\omega-\tilde{E}_{L,j}^\omega|<e^{-L^{a\theta\varrho^\prime}}\}
\end{align*}
satisfies
\begin{equation}\label{appEq0}
 \mathbb{P}[\Omega_L]>1-2L^{-p}-2e^{-L^{a\theta\varrho^\prime}},
\end{equation}
i.e for any configuration in $\Omega_L$, there is an one-to-one correspondence between eigenvalues of $H^\omega_{\Lambda_L}$ and eigenvalues of $\tilde{H}^\omega_{L}$ in $I$
for which center of localization lies in $\Lambda_L\setminus \Lambda_{m_L+\tilde{m}_L}$.
Similarly we can show that for a set of configuration $\tilde{\Omega}_L$ satisfying \eqref{appEq0},
there is a one-to-one correspondence between the eigenvalues of $H^\omega_{\Lambda_L}$ and $\chi_{\Lambda_{m_L+3\tilde{m}_L}}H^\omega_{\Lambda_L}\chi_{\Lambda_{m_L+3\tilde{m}_L}}$ 
for which the center of localization lies in $\Lambda_{m_L+2\tilde{m}_L}$ such that the eigenvalues are exponentially close to each other. 

Observe that, by above process on the configuration space $\Omega_L\cap \tilde{\Omega}_L$, we have exhausted the eigenvalues of $H^\omega_{\Lambda_L}$. 
For $\omega\in \Omega_L\cap \tilde{\Omega}_L$ define
\begin{align*}
S^\omega_L=\{(E^\omega_{L,n},\tilde{E}^\omega_{L,n}): E^\omega_{L,n}\in\sigma(H^\omega_{\Lambda_L})\cap I,\tilde{E}^\omega_{L,n}\in\sigma(\tilde{H}^\omega_{L})\cap I~\&~\eqref{appEq1}~satisfied \}.
\end{align*}
Since the interval $I$ in consideration will be in $(a_j,a_{j-1})$, by definition of $N_j$, we have
\begin{equation}\label{appEq5}
\frac{1}{m_L^{d-j\alpha}}\mathbb{E}^\omega\left[\#\{E\in\sigma(\chi_{\Lambda_{m_L+3\tilde{m}_L}}H^\omega_{\Lambda_L}\chi_{\Lambda_{m_L+3\tilde{m}_L}})\cap I\}\right]=O(1).
\end{equation}
Now for $z\in\mathbb{C}^{+}$ and $\omega\in\Omega_L\cap \tilde{\Omega}_L$, consider 
\begin{align*}
&\mathop{\mathbb{E}}_{\Omega_L\cap \tilde{\Omega}_L}\int_0^1 \left|\xi^\omega_{L,j,t}(Im(\cdot-z)^{-1})-\Xi^\omega_{t,I,L}(Im(\cdot-z)^{-1})\right| dt\\
&=\frac{1}{N_j(I)L^{d-j\alpha}}\mathop{\mathbb{E}}_{\Omega_L\cap \tilde{\Omega}_L}\int_0^1 \left|\sum_{E_n\in\sigma(H^\omega_{\Lambda_L})\cap I} Im \frac{1}{N_{j,I}(E_n)-t-(N_j(I)L^{d-j\alpha})^{-1}z} \right.\\
&\qquad\qquad\qquad \qquad \left. - \sum_{\tilde{E}_n\in\sigma(\tilde{H}^\omega_L)\cap I} Im \frac{1}{N_{j,I}(\tilde{E}_n)-t-(N_j(I)L^{d-j\alpha})^{-1}z} \right| dt \\
&\leq \frac{1}{N_j(I)L^{d-j\alpha}} \mathop{\mathbb{E}}_{\Omega_L\cap \tilde{\Omega}_L}\sum_{(E_n,\tilde{E}_n)\in S_L } \int_0^1 \left| Im \frac{1}{N_{j,I}(E_n)-t-(N_j(I)L^{d-j\alpha})^{-1}z}\right.\\
&\qquad\qquad\qquad\qquad\qquad\qquad\left.- Im \frac{1}{N_{j,I}(\tilde{E}_n)-t-(N_j(I)L^{d-j\alpha})^{-1}z}\right| dt\\
&\qquad\qquad+ \frac{1}{N_j(I)L^{d-j\alpha}} O(m_L^{d-j\alpha}).
\end{align*}
First part converges to zero because of definition of $S^\omega_L$ and the absolute continuity of $N_j$ (from Proposition \ref{pro2}).
For the second part we are integrating over $t$, so we can bound $\int_0^1 Im \frac{1}{N_{j,L}(x)-t-(\beta_{j,L})^{-1}z}dt$ by $\pi$, hence the factor $N_j(I)L^{d-j\alpha}$ outside is unaffected.
So all we have to do is count the eigenvalues of $H^\omega_{\Lambda_L}$ in $I$ for which the center of localization lies in $\Lambda_{m_L+\tilde{m}_L}$, 
which can be taken in account as the eigenvalues of $H^\omega_{\Lambda_{m_L+3\tilde{m}_L}}$, which is given by \eqref{appEq5}.
So using the fact that $Im(\cdot-z)^{-1}$ for $z\in\mathbb{C}^{+}$ are dense in $C_0(\mathbb{R})$, we have
\begin{equation}\label{appEq2}
\mathop{\mathbb{E}}_{\Omega_L\cap \tilde{\Omega}_L}\int_0^1 |\xi^\omega_{L,j,t}(\varphi)-\Xi^\omega_{t,I,L}(\varphi)| dt\xrightarrow{L\rightarrow\infty} 0\qquad a.e~\omega.
\end{equation}
for $\varphi\in C_c(\mathbb{R})$. So we can focus on the point process $\Xi^\omega_{t,I,L}$ because 
\begin{align}
&\left|\mathbb{E}^\omega\int_0^1 e^{-\xi^\omega_{L,j,t}(\varphi)}dt-\mathbb{E}^\omega\int_0^1 e^{-\Xi^\omega_{t,I,L}(\varphi)}dt\right|\nonumber\\
&\qquad\leq\mathop{\mathbb{E}}_{\Omega_L\cap\tilde{\Omega}_L}\int_0^1 |\xi^\omega_{L,j,t}(\varphi)-\Xi^\omega_{t,I,L}(\varphi)|dt+\mathbb{P}[\Omega_L\cap\tilde{\Omega}_L]\xrightarrow{L\rightarrow\infty} 0,\label{appEq6}
\end{align}
for $\varphi\in C_c(\mathbb{R})$ non-negative. Hence all we need to do is show
\begin{equation*}
 \mathbb{E}^\omega\left[\int_0^1 e^{-\Xi^\omega_{t,I,L}(\varphi)}\right]dt\xrightarrow{L\rightarrow\infty} \exp\left(\int(e^{\varphi(x)}-1)dx\right)
\end{equation*}
for $\varphi\in C_c(\mathbb{R})$.
\\\\
\noindent{\bf Step 2:}
This step follows similar steps from Section 3.1 and Section 3.2 from \cite{Klopp}.
For simplicity let us denote 
\begin{equation}\label{appEq7}
 \mathcal{L}^\omega_{I, L}(\varphi)=\int_0^1 e^{-\Xi^\omega_{t,I,L}(\varphi)}dt,
\end{equation}
for $\varphi\in C_c(\mathbb{R})$. 
Since $N_j$ is Lipschitz continuous and is a distribution function on $(a_j,a_{j-1})$ with $N_j(a_{j-1})>N_j(a_{j})$,
following the argument of \cite[Theorem 3.1]{Klopp}, we can partition the set $[0,1]$ as $\cup_{m\in\mathcal{M}}I_m$ where 
\begin{itemize}
 \item $I_m$ are open and $N_{j,I}$ is strictly increasing in $N_{j,I}^{-1}(I_m)$; we denote set of such $m$ by $\mathcal{M}^{+}$,
 \item $I_m$ is singleton and $N_{j,I}$ is constant on $N_{j,I}^{-1}(I_m)$; we denote the set of such $m$ by $\mathcal{M}^{0}$.
\end{itemize}
here we take $I=(a_j,a_{j-1})$.
So with this partition we have
\begin{align}\label{appEq3}
\mathcal{L}^\omega_{I, L}(\varphi)&=\int_0^1 e^{-\Xi^\omega_{t,I,L}(\varphi)}dt\nonumber\\
&=\sum_{m\in\mathcal{M}^{+}}\int_{I_m} e^{-\Xi^\omega_{t,I,L}(\varphi)}dt\nonumber\\
&=\sum_{m\in\mathcal{M}^{+}} N_j(I_m)\int_0^1 e^{-\Xi^\omega_{t,I_m,L}(\varphi)}dt\nonumber\\
&=\sum_{m\in\mathcal{M}^{+}} \frac{N_j(I_m)}{N_j(I)}\mathcal{L}^\omega_{I_m, L}(\varphi)
\end{align}
Hence we only need to show 
\begin{equation*}
\mathbb{E}^\omega[\mathcal{L}^\omega_{I_m, L}(\varphi)]\xrightarrow{L\rightarrow\infty} \exp\left(\int(e^{\varphi(x)}-1)dx\right),
\end{equation*}
for each $m\in\mathcal{M}^{+}$.\\\\
Now let $(\theta,\beta,\gamma)$ be chosen in such a way that the hypothesis of Theorem \ref{infDivThm} is satisfied. 
Let $\{J_{L,i}\}_{i\in G_L}$ be a set of intervals which partitions $I_m$ and satisfies
$$N_j(J_{L,i})\approx L^{-(d-j\alpha)\beta},$$
and set
$$\hat{G}_L=\{i\in G_L:N_j(J_{L,i})\geq |J_{L,i}|^{1+\gamma}\}.$$
Hence for $i\in G_L\setminus \hat{G}_L$ we have $|J_{i,L}|\geq L^{-\frac{\beta(d-j\alpha)}{1+\gamma}}$,
which gives $\# (G_L\setminus \hat{G}_L)=O\left(L^{\frac{\beta(d-j\alpha)}{1+\gamma}}\right)$.
So we get
\begin{equation}\label{appEq4}
 \sum_{i\in G_L\setminus \hat{G}_L} N_j(J_{L,i})=O(L^{\frac{\beta(d-j\alpha)}{1+\gamma} -(d-j\alpha)\beta})\xrightarrow{L\rightarrow \infty}0.
\end{equation}
Following the steps from \eqref{appEq3} for a non-negative function $\varphi\in C_c(\mathbb{R})$ we have
\begin{align}\label{appEq8}
\mathcal{L}^\omega_{I_m, L}(\varphi)=\sum_{i\in\hat{G}_L} \frac{N_j(J_{L,i})}{N_j(I_m)} \mathcal{L}^\omega_{J_{L,i}, L}(\varphi)+O(L^{-\frac{(d-j\alpha)\beta\gamma}{1+\gamma}}),
\end{align}
hence we have to show 
\begin{equation*}
 \sup_{i\in\hat{G}_L}\left|\mathbb{E}^\omega\left[\mathcal{L}^\omega_{J_{L,i}, L}(\varphi)\right]-\exp\left(\int(e^{\varphi(x)}-1)dx\right)\right|\xrightarrow{L\rightarrow \infty}0.
\end{equation*}
\noindent{\bf Step3:}
For the choices $\theta,\beta,\gamma$ as defined earlier and $R> 2d$, there is a choice of $0<\mu<1$ such that Theorem \ref{infDivThm} are valid. 
Let $P_L$ to be the indexing set $\{p_i\}_i$ that is defined in Theorem \ref{infDivThm}, and let  denote 
$$S_L=\{(p_i,E^\omega_{i,n},E^\omega_{k_i,L}): p_i\in P_L, E^\omega_{i,n}\in\sigma(H^\omega_{\Lambda_{l_L}(p_i)}), E^\omega_{k_i,L}\in\sigma(\tilde{H}^\omega_L),
\&~\eqref{infDivThmEq2},\eqref{infDivThmEq3}~are~satisfied\}.$$
Define the point process
$$\hat{\Xi}^\omega_{t,i,L}(\cdot)=\sum_{p\in P_L}\sum_{E\in\sigma(H^\omega_{\Lambda_{l_L}(p)})\cap J_{L,i}}\delta_{N_j(J_{L,i})L^{d-j\alpha} (N_{j,J_{L,i}}(E)-t)}(\cdot),$$
and observe that for $z\in\mathbb{C}^{+}$ we have
\begin{align*}
&\int_0^1 \left|\Xi^\omega_{t,J_{L,i},L}(Im(\cdot-z)^{-1})-\hat{\Xi}^\omega_{t,i,L}(Im(\cdot-z)^{-1})\right|dt\\
&=\frac{1}{N_j(J_{L,i})L^{d-j\alpha}}\int^1_0\left|\sum_{\tilde{E}\in\sigma(\tilde{H}_L^\omega)\cap J_{L,i}}Im\frac{1}{N_j(\tilde{E})-t-(N_j(J_{L,i})L^{d-j\alpha})^{-1}z}\right.\\
&\qquad\qquad\left.-\sum_{p\in P_L}\sum_{E\in\sigma(H^\omega_{\Lambda_{l_L}(p)})\cap J_{L,i}} Im\frac{1}{N_j(E)-t-(N_j(J_{L,i})L^{d-j\alpha})^{-1}z} \right|dt\\
&\leq \frac{1}{N_j(J_{L,i})L^{d-j\alpha}}\int^1_0 \left| \sum_{(p_i,E^\omega_{i,n},E^\omega_{k_i,L})\in S_L} Im\frac{1}{N_j(E^\omega_{i,n})-t-(N_j(J_{L,i})L^{d-j\alpha})^{-1}z}\right.\\
&\qquad\qquad\qquad\qquad\qquad\qquad\left.- Im\frac{1}{N_j(E^\omega_{k_i,L})-t-(N_j(J_{L,i})L^{d-j\alpha})^{-1}z}\right|dt\\
&\qquad+\frac{2\pi}{N_j(J_{L,i})L^{d-j\alpha}} O\left( (L^{d-j\alpha}N_j(J_{L,i}))\left[L^{\eta_1+d\mu-(d-j\alpha)(1-\beta)}
+L^{\eta_2^\prime-(d-j\alpha)(1-\beta)}\right]\right)\\
&\leq O\left((Im z)^{-2}L^{(d-j\alpha)(1-\beta)+d-R}+L^{\eta_1+d\mu-(d-j\alpha)(1-\beta)}+L^{\eta_2^\prime-(d-j\alpha)(1-\beta)}\right)
\xrightarrow{L\rightarrow\infty}0.
\end{align*}
In above, for the eigenvalues which can be approximated well (i.e \eqref{infDivThmEq2} is valid), we are using intermediate value theorem along with the fact that $N_j$ is absolutely continuous. 
For the second part, we are using the fact that integration over $t$ is bounded by $\pi$, so all we need is \eqref{infDivThmEq4} to make estimations.

Since we have freedom to choose $p$ in \eqref{infDivThmEq1} (which comes at a cost of having to choose $L$ large for the conclusion of the Theorem \ref{infDivThm} to hold),
we will choose $p$ to be greater than one.
Hence using the denseness of the functions $Im(\cdot-z)^{-1}$ in $C_0(\mathbb{R})$, we have
\begin{equation}\label{appEq9}
 \sup_{i\in\hat{G}_L}\left|\mathcal{L}^\omega_{J_{L,i}, L}(\varphi)-\int_0^1 e^{-\hat{\Xi}^\omega_{t,i,L}(\varphi)}dt\right|\xrightarrow{L\rightarrow \infty}0\qquad a.e~\omega.
\end{equation}
for non-negative $\varphi\in C_c(\mathbb{R})$. Now observe that
\begin{align*}
&\mathbb{E}^\omega\left[\int_0^1 e^{-\hat{\Xi}^\omega_{t,i,L}(\varphi)}dt\right]\\
&=\int_0^1 \prod_{p\in P_L} \left(\mathbb{E}^\omega\left[\exp\left(-\sum_{E\in\sigma(H^\omega_{\Lambda_{l_L}(p)})\cap J_{L,i}}\varphi(N_j(J_{L,i})L^{d-j\alpha} (N_{j,J_{L,i}}(E)-t)) \right)\right]\right) dt.
\end{align*}
So let us focus on the product,
\begin{align*}
&\sum_{p\in P_L}\ln\left(1-\mathbb{E}^\omega\left[1-e^{-\sum_{E\in\sigma(H^\omega_{\Lambda_{l_L}(p)})\cap J_{L,i}}\varphi(N_j(J_{L,i})L^{d-j\alpha} (N_{j,J_{L,i}}(E)-t))}\right]\right)\\
&=\sum_{p\in P_L}\ln\Bigg(1-\mathbb{P}[\sigma(H^\omega_{\Lambda_{l_L}(p)})\cap J_{L,i}=1] \\
&\qquad\qquad \times \mathop{\mathbb{E}}_{|\sigma(H^\omega_{\Lambda_{l_L}(p)})\cap J_{L,i}|=1}  \left[1-e^{-\varphi(N_j(J_{L,i})L^{d-j\alpha} (N_{j,J_{L,i}}(E^\omega_{\Lambda_{l_L}(p),i})-t))} \right]\Bigg)\\
&\qquad+\sum_{p\in P_L}\ln\Bigg(1-\mathbb{P}[\sigma(H^\omega_{\Lambda_{l_L}(p)})\cap J_{L,i}\geq 2]\\
&\qquad\left.\times \mathop{\mathbb{E}}_{|\sigma(H^\omega_{\Lambda_{l_L}(p)})\cap J_{L,i}|\geq 2}\left[1-e^{-\sum_{E\in\sigma(H^\omega_{\Lambda_{l_L}(p)})\cap J_{L,i}} \varphi(N_j(J_{L,i})L^{d-j\alpha} (N_{j,J_{L,i}}(E)-t)) }\right]\right).
\end{align*}
Here $E^\omega_{\Lambda_{l_L}(p),i}$ denotes the unique eigenvalue of $H^\omega_{\Lambda_{l_L}(p)}$ in $J_{L,i}$.
First we concentrate on the second part of above expression. 
Using \eqref{modEstEq2} we have $\mathbb{P}[\sigma(H^\omega_{\Lambda_{l_L}(p)})\cap J_{L,i}\geq 2]$ vanishing to zero, 
and using the fact that for non-negative $\varphi$, the expression within the expectation is bounded above by one, we have
\begin{align*}
& \sum_{p\in P_L}\Bigg| \ln\Bigg(1-\mathbb{P}[\sigma(H^\omega_{\Lambda_{l_L}(p)})\cap J_{L,i}\geq 2]\\
&\qquad\left.\times\mathop{\mathbb{E}}_{|\sigma(H^\omega_{\Lambda_{l_L}(p)})\cap J_{L,i}|\geq 2}\left[1-e^{-\sum_{E\in\sigma(H^\omega_{\Lambda_{l_L}(p)})\cap J_{L,i}} \varphi(N_j(J_{L,i})L^{d-j\alpha} (N_{j,J_{L,i}}(E)-t))}\right]\right)\Bigg|\\
&\leq \sum_{p\in P_L}\mathbb{P}[\sigma(H^\omega_{\Lambda_{l_L}(p)})\cap J_{L,i}\geq 2]\leq O\left(L^{(d-j\alpha)\left(1-\frac{2\beta}{1+\gamma}+\kappa\right)+2d\mu} \right)\xrightarrow{L\rightarrow\infty}0,
\end{align*}
where last equation follows from \eqref{modEstEq2} and \eqref{infDivThmPfEq3}.

For the first part, since $N_{j,J_{L,i}}$ is distribution function of eigenvalues and we have conditioned over $\sigma(H^\omega_{\Lambda_{l_L}(p)})$ $\cap J_{L,i}$, 
the random variable $N_{j,J_{L,i}}(E^\omega_{\Lambda_{l_L}(p)})$ is uniformly distributed on $[0,1]$ hence the first expression becomes
\begin{align*}
&\sum_{p\in P_L}\ln\Bigg(1-\mathbb{P}[\sigma(H^\omega_{\Lambda_{l_L}(p)})\cap J_{L,i}=1] \\
&\qquad\times\mathbb{E}^\omega\left[1-e^{-\varphi(N_j(J_{L,i})L^{d-j\alpha} (N_{j,J_{L,i}}(E^\omega_{\Lambda_{l_L}(p),i})-t))}\right.\left.\bigg| E^\omega_{\Lambda_{l_L}(p)}\in \sigma(H^\omega_{\Lambda_{l_L}(p)})\cap J_{L,i}\right]\Bigg)\\
&=\sum_{p\in P_L}\ln\bigg(1-\frac{\mathbb{P}[\sigma(H^\omega_{\Lambda_{l_L}(p)})\cap J_{L,i}=1]}{N_j(J_{L,i})L^{d-j\alpha}} \int 1-e^{-\varphi(x)}dx\bigg)\\
&=-\left(\int 1-e^{-\varphi(x)}dx\right)\left(\sum_{p\in P_L}\frac{\mathbb{P}[\sigma(H^\omega_{\Lambda_{l_L}(p)})\cap J_{L,i}=1]}{N_j(J_{L,i})L^{d-j\alpha}}\right) +o(1).
\end{align*}
In above we are using $\ln(1-z)=-z+O(z^2)$ for small enough $z$ and the fact that $\mathbb{P}[\sigma(H^\omega_{\Lambda_{l_L}(p)})\cap J_{L,i}=1]\xrightarrow{L\rightarrow\infty}0$.
The second order term is small follows from the next expression and the fact that 
if we have a doubly indexed sequence $\{\{x_{L,n}\}_{1\leq n\leq L,L\in\mathbb{N}}$ of positive numbers such that $\sum_n x_{L,n}=1$ and $\max_n x_{L,n}\xrightarrow{L\rightarrow\infty} 0$, 
then $\sum_n x_{L,n}^2$ also converges to zero.
Theorem \ref{infDivThm} also implies
\begin{align*}
&\left|\frac{\mathop{\mathbb{E}}_{\mathcal{Z}_L}[\#\{E^\omega_n\in\sigma(H^\omega_L)\cap J_{L,i}\}]}{N_j(J_{L,i})L^{d-j\alpha}}
-\sum_{p\in P_L}\frac{\mathbb{P}[\sigma(H^\omega_{\Lambda_{l_L}(p)})\cap J_{L,i}=1]}{N_j(J_{L,i})L^{d-j\alpha}}\right|\\
&\qquad\qquad\leq O\left(L^{\eta_1+d\mu-(d-j\alpha)(1-\beta)}+L^{\eta^\prime_2-(d-j\alpha)(1-\beta)}\right),
\end{align*}
which combining with the previous steps gives
\begin{equation}\label{appEq10}
 \mathbb{E}^\omega\left[\int_0^1 e^{-\hat{\Xi}^\omega_{t,i,L}(\varphi)}dt\right]\xrightarrow{L\rightarrow\infty}\exp\left(\int (e^{-\varphi(x)}-1)dx\right).
\end{equation}
This completes the proof of the theorem by combining \eqref{appEq6},\eqref{appEq7},\eqref{appEq8},\eqref{appEq9} and \eqref{appEq10}.

\qed
\begin{remark}
The result is weaker than the result proved by Klopp\cite{Klopp}, 
but we have to allow it because for $j$ large we can have $d-j\alpha<1$, hence 
\begin{equation*}
 \sum_{L>M}\mathbb{E}^\omega\left[\left|\int_0^1 e^{-\hat{\Xi}^\omega_{t,i,L}(\varphi)}dt-\exp\left(\int (e^{-\varphi(x)}-1)dx\right)\right|\right]\xrightarrow{M\rightarrow\infty}0.
\end{equation*}
may not hold, and going to $L^p$-space also does not help because part of the estimate is done using probability only.

Since we can choose $\kappa$ as small as we want, we can make sure that $\beta$ is as close to $\frac{1}{2}$.
So if $d-j\alpha>2$ then above series is summable with right choices of $\beta$, and we can show
$$\int_0^1 e^{-\hat{\Xi}^\omega_{t,i,L}(\varphi)}dt\xrightarrow{L\rightarrow\infty}\exp\left(\int (e^{-\varphi(x)}-1)dx\right)\qquad a.e~\omega.$$ 
All other expectations that showed up in the proof can be done with almost everywhere convergence with a little modifications of the calculations done above.
\end{remark}
\subsubsection*{Proof of Theorem \ref{main}:}
\noindent We exploit the idea from Minami\cite{NM} to prove the main result. Divide $\Lambda_L$ into $N_L^d$ numbers of disjoint cubes
$C_p,$ $(p=1,2,\cdots,N_L^d$) with side length $\frac{2L+1}{N_L}$, i.e 
\begin{equation}
\label{disjointunion}
 \Lambda_L=\displaystyle \sqcup_p C_p~~~~with~~|C_p|=\bigg(\frac{2L+1}{N_L}\bigg)^d.
\end{equation}
Define 
$$\partial C_p=\{n\in C_p:\exists~ n'\in \mathbb{Z}^d\setminus C_p~~such~that~|n-n'|=1\}$$
$$int(C_p)=\{n\in C_p:dist(n,\partial C_p)>l_L\}.$$
where $l_L=\delta ~ln L,~\delta>0$ for large large enough $L$ and $N_L\approx(2L+1)^\epsilon,~0<\epsilon<1$. Define the point process $\eta^\omega_{L,p}$ associated to $H^\omega_{C_p}$ by 
\begin{equation}
 \label{tra}
\eta^\omega_{L,p,E}=\sum_{x\in\sigma (H^\omega_{C_p})}\delta_{L^{d-\alpha}(x-E)}.
\end{equation}
Note that the matrices \{$H^\omega_{C_p}\}_p$ are statistically independent.
the proof follows in two steps:
\\\\
{\bf Step 1:} In this step we show that  $\{\xi^\omega_L\}$ and $\{\sum_p\eta^\omega_{L,p,E}\}$ converges to same limit in the sense of distribution.
For this one has to verify:
\begin{equation}
 \label{equiv}
\lim_{L\to\infty}\mathbb{E}^\omega\bigg\{\bigg|e^{-\int fd\xi^\omega_{L,E}}-e^{-\sum_p\int f d\eta^\omega_{L,p,E}}\bigg|\bigg\}=0~~for~~f\in C_c(\mathbb{R}).
\end{equation}
The linear combination of the functions $\phi_z(x)=Im\frac{1}{x-z},~z\in\mathbb{C}^+$ are dense in $L^1(\mathbb{R})$ (for more details see \cite[Appendix: The Stone-Weierstrass Gavotte]{HRWB}). Finally using $|e^{-x}-e^{-y}|\leq |x-y|,~x,y\ge0$ on (\ref{equiv}) for $\phi_z$, one needs to show
\begin{equation}
\label{eqcn}
 \lim_{L\to\infty}\mathbb{E}^\omega\bigg\{\bigg|\int \phi_z d\xi^\omega_{L,E}-\sum_p\int\phi_z d\eta^\omega_{L,p,E}\bigg|\bigg\}=0.
\end{equation}
Since $\Lambda_L$ is the disjoint union of $C_p~(p=1,2,\cdots,N_L^d)$, we have (denote $z_L=E+L^{-(d-\alpha)}z$)
\begin{align}
\label{diff}
\int \phi_z d\xi^\omega_{L,E}-\sum_p\int\phi_z d\eta^\omega_{L,p,E} &= \frac{1}{L^{d-\alpha}}\bigg[
Tr\bigg(Im G^{\Lambda_L}(z_L)\bigg)-\sum_pTr\bigg( Im G^{C_p}(z_L)\bigg)\bigg]\\
&=\frac{1}{L^{d-\alpha}}\sum_p\sum_{n\in C_p}\bigg[ImG^{\Lambda_L}(z_L;n,n)-ImG^{C_p}(z_L;n,n)\bigg].\nonumber
\end{align}
For $z\in\mathbb{C}^+$ and $n\in int(C_p)$, the resolvent equation give us
$$G^{\Lambda_L}(z;n,n)-G^{C_p}(z;n,n)=\sum_{(m,k)\in\partial C_p}G^{\Lambda_L}(z;k,n)~G^{C_p}(z;n,m),$$
here $(m,k)\in\partial C_p$ means $m\in\partial C_p$, $k\in \mathbb{Z}^d\setminus C_p$ with $|m-k|=1$.
Using the above in (\ref{diff})
\begin{align}
 \label{diff1}
&\bigg|\int \phi_z d\xi^\omega_{L,E}-\sum_p\int\phi_z d\eta^\omega_{L,p,E}\bigg| \\
&\qquad \leq \frac{1}{L^{d-\alpha}}\sum_p\sum_{n\in C_p\setminus int(C_p)}
\big[Im G^{\Lambda_L}(z_L;n,n)+ Im G^{C_p}(z_L;n,n)\big]\\
&\qquad \qquad+\frac{1}{L^{d-\alpha}}\sum_p\sum_{n\in int(C_p)}\sum_{(m,k)\in\partial C_p} \big|G^{\Lambda_L}(z_L;k,n)\big| \big|G^{C_p}(z_L;n,m)\big|\nonumber\\
&=A_L+B_L\nonumber.
\end{align}
estimating $A_L$ and $B_L$ as follows
\begin{align}
\label{estA}
 \mathbb{E}^\omega(A_L)&\leq  \frac{2}{L^{d-\alpha}}\sum_p\sum_{n\in C_p\setminus int(C_p)} b_n^{-1},~~(see (\ref{pro1r}))\\
&= O\bigg(L^{-(d-\alpha)}\bigg(\frac{2L+1}{N_L}\bigg)^{d-1-\alpha}N_L^{d-\alpha} l_L\bigg),~~\big(b_n=1+|n|^\alpha\big)\nonumber\\
&=O\big(N_LL^{-1}l_L\big).\nonumber\\
&=O\big(L^{-(1-\epsilon)}ln L\big), ~~N_L=(2L+1)^\epsilon,~~0<\epsilon<1.\nonumber
\end{align}
and
\begin{align}
 \label{estB}
\mathbb{E}^\omega(B_L) &\leq \frac{1}{L^{d-\alpha}}\sum_p\sum_{n\in int(C_p)}\sum_{(m,k)\in\partial C_p}\mathbb{E}^\omega\bigg[ \big|G^{\Lambda_L}(z_L;k,n)\big|^s\big|G^{\Lambda_L}(z_L;k,n)\big|^{(1-s)}\nonumber\\
&\qquad\qquad\qquad\qquad\qquad\qquad\qquad\qquad\qquad \big|G^{C_p}(z_L;n,m)\big|\bigg]\qquad for~0<s<1,
\end{align}
Since $n\in int(C_p)$ and $(m,k)\in\partial C_p$ so we have $|n-k|>l_L$. Using exponential decay of eigenfunctions (\cite[Theorem 1.2 (i)]{GJMS}) we have factional localization estimate $\mathbb{E}^\omega\big(\big|G^\Lambda_L(z_L;k,n)\big|^s\big)\leq C e^{-r |n-k|}$, so using it in (\ref{estB}) gives
\begin{align*}
 \mathbb{E}^\omega(B_L)\leq \frac{C}{L^{d-\alpha}|Imz_L|^{2-s}}N_L^d\bigg(\frac{2L+1}{N_L}\bigg)^d\bigg(\frac{2L+1}{N_L}\bigg)^{d-1}l_Le^{-r l_L}.
\end{align*}
Using $N_L=(2L+1)^\epsilon,~~0<\epsilon<1$, $l_L=\delta ln L~~\delta>0$ and $Imz_L=L^{-(d-\alpha)}z$ in above to get
\begin{equation}
\label{esB}
 \mathbb{E}^\omega(B_L)=O\big(L^{-\tilde{\delta}}ln L\big),
\end{equation}
where $\tilde{\delta}= r\delta-[2d-1+(1-s)(d-\alpha)+\epsilon(1-d)]>0$  with
$$\delta>\frac{1}{r}[2d-1+(1-s)(d-\alpha)+\epsilon(1-d)].$$
Concluding (\ref{eqcn}) from (\ref{diff1}), (\ref{estA}) and (\ref{esB}).
\begin{remark}
\label{equivl}
Defining 
$$\xi^\omega_{L,l_L,E}(\cdot):=\sum_{n\in\Lambda_L\setminus\Lambda_{l_L}}\langle \delta_n,E_{_{H^\omega_{L,l_l}}}\big(E+L^{-(d-\alpha)}(\cdot)\big)\delta_n\rangle\qquad E>a_1,$$
one can follow Step 1 and show that
\begin{equation}
\label{equivl1}
\lim_{L\to\infty}\mathbb{E}^\omega\big(\xi^\omega_{L,l_L,E}(I)\big)=\lim_{L\to\infty}\mathbb{E}^\omega\big(\xi^\omega_{L,E}(I)\big),
\end{equation}
for any bounded interval $I$.
\end{remark}
\noindent {\bf Step 2:} 
In this step, weak convergence of $\displaystyle\sum_{p=1}^{N_L^d}\eta_{L,p,E}$ is shown. 
First we show that the triangular array $\big\{\eta_{L,p,E}:p=1,2,\cdots,N_L^d\big\}_L$ is uniformly asymptotically negligible. For this one need to show
\begin{equation}
 \label{una}
\lim_{L\to \infty}\sup_{p}\mathbb{P}\big(\eta^\omega_{L,p,E}(I)\ge1\big)=0
\end{equation}
for any bounded interval $I$ (see \cite[equation (11.2.2)]{DJ}).
To show $\displaystyle\sum_{p=1}^{N_L^d}\eta_{L,p,E}$ converges weakly to Poisson point process we use \cite[Theorem 11.2.V]{DJ}. For that one need to verify:
\begin{equation}
 \label{con1}
\sum_{p=1}^{N_L^d}\mathbb{P}\big(\eta^\omega_{L,p,E}(I)\ge2\big)=0,~~~~~~as~~L\to\infty
\end{equation}
and 
\begin{equation}
 \label{con2}
\sum_{p=1}^{N_L^d}\mathbb{P}\big(\eta^\omega_{L,p,E}(I)\ge1\big)=N_1'(E)|I|,~~~~as~~~L\to\infty.
\end{equation}
For any bounded interval $I$.\\
The condition (\ref{una}) is obtained using Wegner estimate (\ref{wegner})
\begin{align*}
 \mathbb{P}\big(\eta^\omega_{L,p,E}(I)\ge1\big) & \leq \mathbb{E}^\omega\big[\eta^\omega_{L,p,E}(I)\big]\\
                                                &= \mathbb{E}^\omega\big[TrE_{H^\omega_{C_p}}\big(E+L^{-(d-\alpha)}I\big)\big]\\
                                                &\leq\pi\frac{|I|}{L^{d-\alpha}}\sum_{n\in C_p}(1+|n|^\alpha)^{-1}\\
                                                &=O\bigg(L^{-(d-\alpha)}\bigg(\frac{2L+1}{N_L}\bigg)^{d-\alpha}\bigg)=O(N_L^{-(d-\alpha)})
\end{align*}
giving  (\ref{una}). Next, note that
\begin{equation}
 \label{formula}
\mathbb{P}\big(\eta^\omega_{L,p,E}(I)\ge1\big)=\mathbb{E}^\omega\big[\eta^\omega_{L,p,E}(I)\big]-\sum_{j\ge2}\mathbb{P}\big(\eta^\omega_{L,p,E}(I)\ge j\big).
\end{equation}
Using Minami estimate (\ref{Minami}) for second term of above equation,
\begin{align}
\label{m1}
 \sum_{j\ge2}\mathbb{P}\big(\eta^\omega_{L,p,E}(I)\ge j\big) &=\sum_{j\ge2} (j-1)\mathbb{P}\big(\eta^\omega_{L,p,E}(I)=j\big)\\
&\leq \sum_{j\ge2} j(j-1)\mathbb{P}\big(\eta^\omega_{L,p,E}(I)=j\big)\nonumber\\
&=\mathbb{E}^\omega\big[\eta^\omega_{L,p,E}(I)\big(\eta^\omega_{L,p,E}(I)-1\big)\big]\nonumber\\
&=\mathbb{E}^\omega\bigg[TrE_{H^\omega_{C_p}}\big(E+L^{-(d-\alpha)}I\big)
\bigg\{TrE_{H^\omega_{C_p}}\big(E+L^{-(d-\alpha)}I\big)-1\bigg\}\bigg]\nonumber\\
&=\bigg(\sum_{n\in C_p}(1+|n|^\alpha)^{-1}\frac{|I|}{L^{d-\alpha}}\bigg)^2\nonumber\qquad\qquad(\text{using}~(\ref{Minami}))\\
&=O\bigg(\bigg(\bigg(\frac{2L+1}{N_L}\bigg)^{d-\alpha}L^{-(d-\alpha)}\bigg)^2\bigg)=O\bigg(N_L^{-2(d-\alpha)}\bigg).\nonumber
\end{align}
Hence
\begin{equation}
 \label{m2}
\sum_p\sum_{j\ge2}\mathbb{P}\big(\eta^\omega_{L,p,E}(I)\ge j\big)=O\bigg(N_L^{-(d-\alpha)}\bigg).
\end{equation}
(\ref{con1}) follows from this. To get \eqref{con2}, Remark \ref{ana-cont} and Remark \ref{equivl} is used 
\begin{align}
 \label{intensity}
\lim_{L\to\infty}\sum_p\mathbb{E}^\omega\left[\eta^\omega_{L,p,E}(I)\right]&=\lim_{L\to\infty}\mathbb{E}^\omega\left[\xi^\omega_{L,E}(I)\right]\\
                              &=\lim_{L\to\infty}\mathbb{E}^\omega\left(\xi^\omega_{L,l_L,E}\right), \qquad(\text{using}~(\ref{equivl1}))\nonumber\\
                              &=\lim_{L\to\infty}L^{d-\alpha}\int_{E+L^{-(d-\alpha)}I}N^{'}_{L,l_L}(x)dx\nonumber\\
                              &=\lim_{L\to\infty}\int_{I}N^{'}_{L,l_L}\big(E+L^{-(d-\alpha)}y\big)dy=N_1'(E)|I|,\qquad(\text{using}~(\ref{uni-den})).\nonumber
\end{align}

Using (\ref{m2}) and (\ref{intensity}) in (\ref{formula}) we get (\ref{con2}) completing the proof.

\qed
\\
\noindent{\bf Acknowledgement:}  
We would like to thank the referee for his suggestions to improve the article. 
The authors are thankful to M. Krishna for going through the manuscript.


\end{document}